\documentclass[12pt,a4paper]{article}
\usepackage{amssymb}
\usepackage{amsmath}
\usepackage{amsthm}
\usepackage{indentfirst}
\usepackage{ot1patch}
\usepackage{psfrag}
\usepackage{rotate}

\newtheorem{pr}{Proposition}
\newtheorem{lm}{Lemma}

\DeclareMathOperator{\Sqf}{Sqf}
\DeclareMathOperator{\lcm}{lcm}

\newcommand{\rpr}{\mbox{\small\rm rpr}}

\author{Piotr J\k{e}drzejewicz, \L ukasz Matysiak, Janusz Zieli\'nski}
\title{A note on square-free factorizations}
\date{}

\begin{document}

\maketitle

\begin{abstract}
We analyze properties of various square-free factorizations
in greatest common divisor domains and domains satisfying
the ascending chain condition for principal ideals.
\end{abstract}

\begin{table}[b]\footnotesize\hrule\vspace{1mm}
\end{table}

\section*{Introduction}

Throughout this article by a ring we mean a commutative ring with unity.
By a domain we mean a ring without zero divisors.
By $R^{\ast}$ we denote the set of all invertible elements of a ring $R$.
Given elements $a,b\in R$, we write $a\sim b$ if $a$ and $b$ are associated,
and $a\mid b$ if $b$ is divisible by $a$.
Furthermore, we write $a\:\rpr\,b$ if $a$ and $b$ are relatively prime,
that is, have no common non-invertible divisors.
If $R$ is a ring, then
by $\Sqf R$ we denote the set of all square-free elements of $R$,
where an element $a\in R$ is called square-free if it can not
be presented in the form $a=b^2c$ with $b\in R\setminus R^{\ast}$,
$c\in R$.

\medskip

In \cite{factorial} we discuss many factorial properties of subrings,
in particular involving square-free elements.
The aim of this paper is to collect various ways to present
an element as a product of square-free elements and to study
the existence and uniqueness questions in larger classes than
the class of unique factorization domains.
In Proposition~\ref{p1} we obtain the equivalence of factorizations
(ii) -- (vii) for GCD-domains.
It appears that some preparatory properties hold in a more
larger class, namely pre-Schreier domains (Lemma~\ref{l2}).
We also prove the existence of factorizations (i) -- (iii)
in Proposition~\ref{p1} for ACCP-domains,
but their uniqueness we obtain in Proposition~\ref{p2} for GCD-domains.
We refer to Clark's survey article \cite{Clark} for more information
about GCD-domains and ACCP-domains.

\section{Preliminary lemmas}
\label{s1}

Note the following easy lemma.

\begin{lm}
\label{l1}
Let $R$ be a ring.
If $a\in\Sqf R$ and $a=b_1b_2\dots b_n$, then
$b_1,b_2,\dots,b_n\in\Sqf R$ and $b_i\:\rpr\,b_j$ for $i\neq j$.
\end{lm}

Recall from \cite{Zafrullah} that a domain $R$ is called
pre-Schreier if every non-zero element $a\in R$ is primal,
that is, for every $b,c\in R$ such that $a\mid bc$
there exist $a_1,a_2\in R$ such that $a=a_1a_2$,
$a_1\mid b$ and $a_2\mid c$.

\begin{lm}
\label{l2}
Let $R$ be a pre-Schreier domain.

\medskip

\noindent
{\bf a)}
Let $a,b,c\in R$, $a\neq 0$.
If $a\mid bc$ and $a\:\rpr\,b$, then $a\mid c$.

\medskip

\noindent
{\bf b)}
Let $a,b,c,d\in R$.
If $ab=cd$, $a\:\rpr\,c$ and $b\:\rpr\,d$, then $a\sim d$
and $b\sim c$.

\medskip

\noindent
{\bf c)}
Let $a,b,c\in R$.
If $ab=c^2$ and $a\:\rpr\,b$, then there exist $c_1,c_2\in R$
such that $a\sim c_1^2$, $b\sim c_2^2$ and $c=c_1c_2$.

\medskip

\noindent
{\bf d)}
Let $a_1,\dots,a_n,b\in R$.
If $a_i\:\rpr\,b$ for $i=1,\dots,n$, then $a_1\ldots a_n\:\rpr\,b$.

\medskip

\noindent
{\bf e)}
Let $a_1,\dots,a_n\in R$.
If $a_1,\dots,a_n\in\Sqf R$ and $a_i\:\rpr\,a_j$ for all $i\neq j$,
then $a_1\ldots a_n\in\Sqf R$.
\end{lm}

\begin{proof}
{\bf a)}
If $a\mid bc$, then $a=a_1a_2$ for some $a_1,a_2\in R\setminus\{0\}$
such that $a_1\mid b$ and $a_2\mid c$.
If, moreover, $a\:\rpr\,b$, then $a_1\in R^{\ast}$.
Hence, $a\sim a_2$, so $a\mid c$.

\medskip

\noindent
{\bf b)}
Assume that $ab=cd$, $a\:\rpr\,c$ and $b\:\rpr\,d$.
If $a=0$ and $R$ is not a field, then $c\in R^{\ast}$,
so $d=0$ and then $b\in R^{\ast}$.
Now, let $a,d\neq 0$.

\medskip

Since $a\mid cd$ and $a\:\rpr\,c$, we have $a\mid d$ by a).
Similarly, since $d\mid ab$ and $d\:\rpr\,b$, we obtain $d\mid a$.
Hence, $a\sim d$, and then $b\sim c$.

\medskip

\noindent
{\bf c)}
Let $ab=c^2$ and $a\:\rpr\,b$.
Since $c\mid ab$, there exist $c_1,c_2\in R\setminus\{0\}$
such that $c_1\mid a$, $c_2\mid b$ and $c=c_1c_2$.
Hence, $a=c_1d$ and $b=c_2e$ for some $d,e\in R$,
and we obtain $de=c_1c_2$.
We have $d\:\rpr\,c_2$, because $d\mid a$ and $c_2\mid b$,
analogously $e\:\rpr\,c_1$, so $d\sim c_1$ and $e\sim c_2$, by b).
Finally, $a\sim c_1^2$, $b\sim c_2^2$.

\medskip

\noindent
{\bf d)}
Induction.
Let $a_i\:\rpr\,b$ for $i=1,\dots,n+1$.
Put $a=a_1\ldots a_n$.
Assume that $a\:\rpr\,b$.
Let $c\in R\setminus\{0\}$ be a common divisor of $aa_{n+1}$ and $b$.
Since $c\mid aa_{n+1}$, there exist $c_1,c_2\in R\setminus\{0\}$
such that $c_1\mid a$, $c_2\mid a_{n+1}$ and $c=c_1c_2$.
We see that $c_1,c_2\mid b$, so $c_1,c_2\in R^{\ast}$,
and then $c\in R^{\ast}$.

\medskip

\noindent
{\bf e)}
Induction.
Take $a_1,\dots,a_{n+1}\in\Sqf R$ such that $a_i\:\rpr\,a_j$
for $i\neq j$.
Put $a=a_1\ldots a_n$.
Assume that $a\in\Sqf R$.
Let $aa_{n+1}=b^2c$ for some $b,c\in R\setminus\{0\}$.

\medskip

Since $c\mid aa_{n+1}$, there exist $c_1,c_2\in R\setminus\{0\}$
such that $c=c_1c_2$, $c_1\mid a$ and $c_2\mid a_{n+1}$,
so $a=c_1d$ and $a_{n+1}=c_2e$, where $d,e\in R$.
We obtain $de=b^2$.
By d) we have $a\:\rpr\,a_{n+1}$, so $d\:\rpr\,e$.
And then by c), there exist $b_1,b_2\in R$
such that $d\sim b_1^2$, $e\sim b_2^2$ and $b=b_1b_2$.
Since $a,a_{n+1}\in\Sqf R$, we infer $b_1,b_2\in R^{\ast}$,
so $b\in R^{\ast}$.
\end{proof}

\section{Square-free factorizations}
\label{s2}

In Proposition~\ref{p1} below we collect possible presentations
of an element as a product of square-free elements or their powers.
We distinct presentations (ii) and (iii), presentations (iv) and (v),
and presentations (vi) and (vii),
because (ii), (iv) and (vi) are of a simpler form,
but in (iii), (v) and (vii) the uniqueness will be more natural
(in Proposition~\ref{p2}).

\medskip

Recall that a domain $R$ is called a GCD-domain if the intersection
of any two principal ideals is a principal ideal.
Every GCD-domain is pre-Schreier (\cite{Cohn}, Theorem~2.4).
Note that in a GCD-domain LCMs exist (\cite{Cohn}, Theorem~2.1).
Recall also that a domain $R$ is called an ACCP-domain
if it satisfies the ascending chain condition for principal ideals.

\begin{pr}
\label{p1}
Let $R$ be a ring.
Given a non-zero element $a\in R\setminus R^{\ast}$,
consider the following conditions:

\medskip

\noindent
{\rm (i)} \
there exist $b\in R$ and $c\in\Sqf R$ such that $a=b^2c$,

\medskip

\noindent
{\rm (ii)} \
there exist $n\geqslant 0$ and $s_0,s_1,\dots,s_n\in\Sqf R$
such that $a=s_n^{2^n}s_{n-1}^{2^{n-1}}\dots s_1^2s_0$,

\medskip

\noindent
{\rm (iii)} \
there exist $n\geqslant 1$,
$s_1,s_2,\dots,s_n\in (\Sqf R)\setminus R^{\ast}$,
$k_1<k_2<\ldots<k_n$, $k_1\geqslant 0$,
and $c\in R^{\ast}$ such that
$a=cs_n^{2^{k_n}}s_{n-1}^{2^{k_{n-1}}}\dots s_2^{2^{k_2}}
s_1^{2^{k_1}}$,

\medskip

\noindent
{\rm (iv)} \
there exist $n\geqslant 1$ and $s_1,s_2,\dots,s_n\in\Sqf R$
such that $s_i\mid s_{i+1}$ for $i=1,\dots,n-1$,
and $a=s_1s_2\dots s_n$,

\medskip

\noindent
{\rm (v)} \
there exist $n\geqslant 1$,
$s_1,s_2,\dots,s_n\in (\Sqf R)\setminus R^{\ast}$,
$k_1,k_2,\ldots,k_n\geqslant 1$, and $c\in R^{\ast}$ such that
$s_i\mid s_{i+1}$ and $s_i\not\sim s_{i+1}$ for $i=1,\dots,n-1$,
and $a=cs_1^{k_1}s_2^{k_2}\dots s_n^{k_n}$,

\medskip

\noindent
{\rm (vi)} \
there exist $n\geqslant 1$ and $s_1,s_2,\dots,s_n\in\Sqf R$
such that $s_i\:\rpr\,s_j$ for $i\neq j$,
and $a=s_1s_2^2s_3^3\ldots s_n^n$,

\medskip

\noindent
{\rm (vii)} \
there exist $n\geqslant 1$,
$s_1,s_2,\dots,s_n\in (\Sqf R)\setminus R^{\ast}$,
$k_1<k_2<\ldots<k_n$, $k_1\geqslant 1$, and $c\in R^{\ast}$
such that $s_i\:\rpr\,s_j$ for $i\neq j$,
and $a=cs_1^{k_1}s_2^{k_2}\dots s_n^{k_n}$.

\medskip

\noindent
{\bf a)}
In every ring $R$ the following holds:
$${\rm (i)}\Leftarrow {\rm (ii)}\Leftrightarrow {\rm (iii)},
\quad\quad {\rm (iv)}\Leftrightarrow {\rm (v)}\Rightarrow
{\rm (vi)}\Leftrightarrow {\rm (vii)}.$$

\medskip

\noindent
{\bf b)}
If $R$ is a GCD-domain, then all conditions {\rm (ii)} -- {\rm (vii)}
are equivalent.

\medskip

\noindent
{\bf c)}
If $R$ is a ACCP-domain, then conditions {\rm (i)} -- {\rm (iii)}
hold.

\medskip

\noindent
{\bf d)}
If $R$ is a UFD, then all conditions {\rm (i)} -- {\rm (vii)} hold.
\end{pr}

\begin{proof}
\noindent
{\bf a)}
Implication ${\rm (i)}\Leftarrow {\rm (ii)}$ and equivalencies
${\rm (ii)}\Leftrightarrow {\rm (iii)}$,
${\rm (iv)}\Leftrightarrow {\rm (v)}$,
${\rm (vi)}\Leftrightarrow {\rm (vii)}$
are obvious, so it is enough to prove implication
${\rm (iv)}\Rightarrow {\rm (vi)}$.

\medskip

Assume that $a=s_1s_2\dots s_n$, where $s_1,s_2,\dots,s_n\in\Sqf R$
and $s_i\mid s_{i+1}$ for $i=1,\dots,n-1$.
Let $s_{i+1}=s_it_{i+1}$, where $t_{i+1}\in R$, for $i=1,\dots,n-1$.
Put also $t_1=s_1$.
Then $s_i=t_1t_2\dots t_i$ for each $i$.
Since $s_n\in\Sqf R$, by Lemma~\ref{l1} we obtain
that $t_1,t_2,\dots,t_n\in\Sqf R$ and $t_i\:\rpr\,t_j$ for $i\neq j$.
Morover, we have $s_1s_2\dots s_n=t_1^nt_2^{n-1}\dots t_n$.

\medskip

\noindent
{\bf b)}
Let $R$ be a GCD-domain.

\medskip

\noindent
${\rm (vi)}\Rightarrow {\rm (iv)}$ \
Assume that $a=s_1s_2^2s_3^3\ldots s_n^n$,
where $s_1,s_2,\dots,s_n\in\Sqf R$
and $s_i\:\rpr\,s_j$ for $i\neq j$.
We see that
$$s_1s_2^2s_3^3\ldots s_n^n=
s_n(s_ns_{n-1})(s_ns_{n-1}s_{n-2})\dots
(s_ns_{n-1}\dots s_2)(s_ns_{n-1}\dots s_2s_1).$$
Since $R$ is a GCD-domain, $s_ns_{n-1}\dots s_i\in\Sqf R$
for each $i$ by Lemma~\ref{l2} e).

\medskip

\noindent
${\rm (vi)}\Rightarrow {\rm (ii)}$ \
Let $a=s_1s_2^2s_3^3\ldots s_n^n$, where $s_1,s_2,\dots,s_n\in\Sqf R$,
and $s_i\:\rpr\,s_j$ for $i\neq j$.
For every $k\in\{1,2,\dots,n\}$ put $k=\sum_{i=0}^r c_i^{(k)} 2^i$,
where $c_i^{(k)}\in\{0,1\}$.
Then
$$a=\prod_{k=1}^ns_k^k=
\prod_{k=1}^ns_k^{\sum_{i=0}^r c_i^{(k)} 2^i}=
\prod_{k=1}^n\prod_{i=0}^r s_k^{c_i^{(k)} 2^i}=
\prod_{i=0}^r\big(\prod_{k=1}^n s_k^{c_i^{(k)}}\big)^{2^i},$$
where $\prod_{k=1}^n s_k^{c_i^{(k)}}\in\Sqf R$
for each $i$ by Lemma~\ref{l2} e).

\medskip

\noindent
${\rm (ii)}\Rightarrow {\rm (vi)}$ \
Let $a=s_n^{2^n}s_{n-1}^{2^{n-1}}\dots s_1^2s_0$,
where $s_0,s_1,\dots,s_n\in\Sqf R$.
For every $k\in\{1,2,\dots,2^{n+1}-1\}$ put
$k=\sum_{i=0}^n c_i^{(k)} 2^i$, where $c_i^{(k)}\in\{0,1\}$.
Let $t_k'=\gcd(s_i\colon c_i^{(k)}=1)$,
$t_k''=\lcm(s_i\colon c_i^{(k)}=0)$
and $t_k'=\gcd(t_k',t_k'')\cdot t_k$, where $t_k\in R$.
Then $t_k$ is the greatest among these common divisors
of all $s_i$ such that $c_i^{(k)}=1$,
which are relatively prime to all $s_i$ such that $c_i^{(k)}=0$.
In particular,
$t_k\mid s_i$ for every $k,i$ such that $c_i^{(k)}=1$,
and $t_k\:\rpr\,s_i$ for every $k,i$ such that $c_i^{(k)}=0$.
In each case, $\gcd(s_i,t_k)=t_k^{c_i^{(k)}}$.
Moreover, $t_k\:\rpr\,t_l$ for every $k\neq l$.

\medskip

Since $s_i\mid t_1t_2\dots t_{2^{n+1}-1}$, we obtain
$$s_i=\gcd(s_i,\prod_{k=1}^{2^{n+1}-1}t_k)=
\prod_{k=1}^{2^{n+1}-1}\gcd(s_i,t_k)=
\prod_{k=1}^{2^{n+1}-1}t_k^{c_i^{(k)}},$$
so
$$\prod_{i=0}^n(s_i)^{2^i}=
\prod_{i=0}^n\prod_{k=1}^{2^{n+1}-1}\big(t_k^{c_i^{(k)}}\big)^{2^i}=
\prod_{k=1}^{2^{n+1}-1}\prod_{i=0}^nt_k^{c_i^{(k)} 2^i}=
\prod_{k=1}^{2^{n+1}-1}t_k^{\sum_{i=0}^n c_i^{(k)} 2^i}=
\prod_{k=1}^{2^{n+1}-1}t_k^k.$$

Moreover, $t_k\in\Sqf R$, because for $k\in\{1,2,\dots,2^{n+1}-1\}$
there exists $i$ such that $c_i^{(k)}=1$, and then $t_k\mid s_i$.

\medskip

\noindent
{\bf c)}
Let $R$ be an ACCP-domain.
In this proof we follow the idea of the second proof
of Proposition 9 from \cite{Clark}, p.\ 7, 8.

\medskip

\noindent
{\rm (i)} \
If $a\not\in\Sqf R$, then $a=b_1^2c_1$, where
$b_1\in R\setminus R^{\ast}$, $c_1\in R$.
If $c_1\not\in\Sqf R$, then $c_1=b_2^2c_2$, where
$b_2\in R\setminus R^{\ast}$, $c_2\in R$.
Repeating this process, we obtain a strongly ascending
chain of principal ideals
$Ra\subsetneqq Rc_1\subsetneqq Rc_2\subsetneqq \ldots$,
so for some $k$ we will have $c_{k-1}=b_k^2c_k$,
$b_k\in R\setminus R^{\ast}$, and $c_k\in\Sqf R$.
Then $a=(b_1\ldots b_k)^2c_k$.

\medskip

\noindent
{\rm (iii)} \
If $a\not\in\Sqf R$, then by {\rm (i)} there exist
$a_1\in R\setminus R^{\ast}$ and $s_0\in\Sqf R$
such that $a=a_1^2s_0$.
If $a_1\not\in\Sqf R$, then again, by {\rm (i)}
there exist $a_2\in R\setminus R^{\ast}$ and $s_1\in\Sqf R$
such that $a_1=a_2^2s_1$.
Repeating this process, we obtain a strongly ascending
chain of principal ideals
$Ra\subsetneqq Ra_1\subsetneqq Ra_2\subsetneqq \ldots$,
so for some $k$ we will have $a_{k-1}=a_k^2s_{k-1}$,
$a_k\in (\Sqf R)\setminus R^{\ast}$, $s_{k-1}\in\Sqf R$.
Putting $s_k=a_k$ we obtain:
$$a=a_1^2s_0=a_2^{2^2}s_1^2s_0=\ldots=
s_n^{2^n}\ldots s_2^{2^2}s_1^2s_0.$$

\medskip

\noindent
{\bf d)}
This is a standard fact following from the irreducible
decomposition.
\end{proof}

\section{The uniqueness of factorizations}
\label{s3}

The following proposition concerns the uniqueness of square-free
decompositions from Proposition~\ref{p1}.
In {\rm (i)} -- {\rm (iii)} we assume that $R$ is a GCD-domain,
in {\rm (iv)} -- {\rm (vii)} $R$ is a UFD.

\begin{pr}
\label{p2}
{\rm (i)} \
Let $b,d\in R$ and $c,e\in\Sqf R$.
If $$b^2c=d^2e,$$ then $b\sim d$ and $c\sim e$.

\medskip

\noindent
{\rm (ii)} \
Let $s_0,s_1,\dots,s_n\in\Sqf R$ and $t_0,t_1,\dots,t_m\in\Sqf R$,
$n\leqslant m$.
If $$s_n^{2^n}s_{n-1}^{2^{n-1}}\dots s_1^2s_0=
t_m^{2^m}t_{m-1}^{2^{m-1}}\dots t_1^2t_0,$$
then $s_i\sim t_i$ for $i=0,\dots,n$ and, if $m>n$,
then $t_i\in R^{\ast}$ for $i=n+1,\dots,m$.

\medskip

\noindent
{\rm (iii)} \
Let $s_1,s_2,\dots,s_n\in (\Sqf R)\setminus R^{\ast}$,
$t_1,t_2,\dots,t_m\in (\Sqf R)\setminus R^{\ast}$,
$k_1<k_2<\ldots<k_n$, $l_1<l_2<\ldots<l_m$ and $c,d\in R^{\ast}$.
If $$cs_n^{2^{k_n}}s_{n-1}^{2^{k_{n-1}}}\dots s_2^{2^{k_2}}
s_1^{2^{k_1}}=dt_m^{2^{l_m}}t_{m-1}^{2^{l_{m-1}}}\dots t_2^{2^{l_2}}
t_1^{2^{l_1}},$$ then $n=m$, $s_i\sim t_i$ and $k_i=l_i$
for $i=1,\dots,n$.

\medskip

\noindent
{\rm (iv)} \
Let $s_1,s_2,\dots,s_n\in\Sqf R$, $t_1,t_2,\dots,t_m\in\Sqf R$,
$n\leqslant m$, $s_i\mid s_{i+1}$ for $i=1,\dots,n-1$,
and $t_i\mid t_{i+1}$ for $i=1,\dots,m-1$.
If $$s_1s_2\dots s_n=t_1t_2\dots t_m,$$
then $s_i\sim t_{i+m-n}$ for $i=1,\dots,n$
and, if $m>n$, then $t_i\in R^{\ast}$ for $i=1,\dots,m-n$.

\medskip

\noindent
{\rm (v)} \
Let $s_1,s_2,\dots,s_n\in (\Sqf R)\setminus R^{\ast}$,
$t_1,t_2,\dots,t_m\in (\Sqf R)\setminus R^{\ast}$,
$k_1,k_2,\ldots,k_n$ $\geqslant 1$, $l_1,l_2,\ldots,l_m\geqslant 1$,
$c,d\in R^{\ast}$,
$s_i\mid s_{i+1}$ and $s_i\not\sim s_{i+1}$ for $i=1,\dots,n-1$,
$t_i\mid t_{i+1}$ and $t_i\not\sim t_{i+1}$ for $i=1,\dots,m-1$.
If $$cs_1^{k_1}s_2^{k_2}\dots s_n^{k_n}=
dt_1^{l_1}t_2^{l_2}\dots t_m^{l_m},$$
then $n=m$, $s_i\sim t_i$ and $k_i=l_i$ for $i=1,\dots,n$.

\medskip

\noindent
{\rm (vi)} \
Let $s_1,s_2,\dots,s_n\in\Sqf R$, $t_1,t_2,\dots,t_m\in\Sqf R$,
$n\leqslant m$, $s_i\:\rpr\,s_j$ for $i\neq j$
and $t_i\:\rpr\,t_j$ for $i\neq j$.
If $$s_1s_2^2s_3^3\ldots s_n^n=t_1t_2^2t_3^3\ldots t_m^m,$$
then $s_i\sim t_i$ for $i=1,\dots,n$ and, if $m>n$,
then $t_i\in R^{\ast}$ for $i=n+1,\dots,m$.

\medskip

\noindent
{\rm (vii)} \
Let $s_1,s_2,\dots,s_n\in (\Sqf R)\setminus R^{\ast}$,
$t_1,t_2,\dots,t_m\in (\Sqf R)\setminus R^{\ast}$,
$1\leqslant k_1<k_2<\ldots<k_n$, $1\leqslant l_1<l_2<\ldots<l_m$,
$c,d\in R^{\ast}$, $s_i\:\rpr\,s_j$ for $i\neq j$,
and $t_i\:\rpr\,t_j$ for $i\neq j$.
If $$cs_1^{k_1}s_2^{k_2}\dots s_n^{k_n}=
dt_1^{l_1}t_2^{l_2}\dots t_m^{l_m},$$
then $n=m$, $s_i\sim t_i$ and $k_i=l_i$ for $i=1,\dots,n$.
\end{pr}

\begin{proof}
{\rm (i)} \
Assume that $b^2c=d^2e$.
Put $f=\gcd(b,d)$, $g=\gcd(c,e)$, $b=fb_0$, $d=fd_0$,
$c=gc_0$, and $e=ge_0$, where $b_0,c_0,d_0,e_0\in R$.
We obtain $b_0^2c_0=d_0^2e_0$, $\gcd(c_0,e_0)=1$
and $\gcd(b_0,d_0)=1$, so also $\gcd(b_0^2,d_0^2)=1$.
By Lemma~\ref{l2} b), we infer $b_0^2\sim e_0$
and $c_0\sim d_0^2$, but $c_0,e_0\in\Sqf R$ by Lemma~\ref{l1},
so $b_0,d_0\in R^{\ast}$, and then $c_0,e_0\in R^{\ast}$.

\medskip

\noindent
Statements {\rm (ii)}, {\rm (iii)} follow from {\rm (i)}.

\medskip

\noindent
Statements {\rm (iv)} -- {\rm (vii)} are straightforward using an
irreducible decomposition.
\end{proof}

\end{document}